\documentclass[12pt]{amsart} \usepackage{amssymb}

\newcommand{\wrlab}[1]{\label{#1}}

\newtheorem{thm}{THEOREM}[section]

\newtheorem{lem}[thm]{LEMMA} 
\newtheorem{cor}[thm]{COROLLARY} \newtheorem{prop}[thm]{PROPOSITION}
 
 \newtheorem{thm*}{THEOREM}[]

\newcommand{\tref}[1]{Theorem~\ref{#1}}
\newcommand{\cref}[1]{Corollary~\ref{#1}}
\newcommand{\pref}[1]{Proposition~\ref{#1}}

\newcommand{\lref}[1]{Lemma~\ref{#1}}

\def\R{{\mathbb R}}  \def\T{{\mathbb T}}

\def\scru{{\mathcal U}}  \def\scrw{{\mathcal W}}

\font\tenolde=eufm10 at 10pt
\font\sevenolde=eufm7
\font\fiveolde=eufm5
\newfam\oldefam
\textfont\oldefam=\tenolde
\scriptfont\oldefam=\sevenolde
\scriptscriptfont\oldefam=\fiveolde

  \def\Lg{{\mathfrak g}}

\def\exp{\hbox{exp}}

\def\Stab{\hbox{Stab}}

\begin{document}
\bibliographystyle{alpha}

\title[Moving frames, prolongations of real algebraic actions]
{MOVING FRAMES AND PROLONGATIONS OF~REAL~ALGEBRAIC~ACTIONS}


\author{Scot Adams}
\address{School of Mathematics\\ University of Minnesota\\Minneapolis, MN 55455
\\ adams@math.umn.edu}

\date{June 11, 2013\qquad Printout date: \today}

\begin{abstract}
We prove that every real algebraic action on a smooth real algebraic variety
has a prolongation with a ``moving frame''.
\end{abstract}

\maketitle

 
\section{Introduction\wrlab{sect-intro}}

We prove (\pref{prop-free-proper-in-prolong} and
\cref{cor-free-proper-in-prolong}) that every real algebraic action on
a smooth real algebraic variety has a prolongation with a ``moving frame''.
That is, there is a nonempty invariant open subset of a jet bundle
on which the action is free and proper.

The basic theme here is well-understood: Real algebraic actions are
inherently uncomplicated from a dynamical point of view.
The type of recurrent behavior that often appears in smooth dynamics
is impossible to replicate in the algebraic setting.
To illustrate, fix an irrational flow on the two-dimensional torus $\T^2$.
Such a flow is real analytic, but {\it not} algebraic.
One of the orbits is the image of an injective homomorphism $\iota:\R\to\T^2$.
This map $\iota$ is nonproper in a very strong sense:
For {\it any} nonempty open subset $U$ of~$\T^2$,
its preimage $\iota^{-1}(U)$ is not precompact in $\R$.
This ``strong nonproperness'' persists to jet bundles and Cartesian powers,
so, generally, we cannot hope, in the real analytic category,
to prove the existence of a moving frame in such prolongations.

The situation is very different for algebraic actions.
For example, any nonconstant polynomial function $\R\to\R$ is proper.
In general, an injective real algebraic map is not always proper,
but comes close: Let
$$Z\qquad:=\qquad\{\,(x,y)\in\R^2\,\,|\,\,xy=1\,\}.$$
Then $(x,y)\mapsto x:Z\to\R$ is nonproper,
but even this map is ``somewhat proper''
in the sense that there is a nonempty open subset of $\R$
whose preimage is precompact in~$Z$.
\lref{lem-alg-implies-ae-proper} tells us that this ``somewhat properness''
appears in any real algebraic map with compact fibers.
\lref{lem-alg-implies-ae-proper} leads to \lref{lem-cptstab-proper}
which asserts that properness is ubiquitous in any real algebraic action
with compact stabilizers.
Note that \lref{lem-alg-implies-ae-proper} is also used to develop freeness:
See the proof of \lref{lem-free-on-open},
which depends on \lref{lem-force-bdd-punct-stabs},
which, in turn, depends on \lref{lem-alg-implies-ae-proper}.

The other main algebraic tool (besides \lref{lem-alg-implies-ae-proper} and \lref{lem-cptstab-proper})
that is used in this note is the Descending Chain Condition on real algebraic subgroups.
This DCC property is used to establish the existence of trivial stabilizers
in prolongation: See the proof of \lref{lem-triv-stab-in-prolong}.

The reader interested in learning more about moving frames may consult the Wikipedia entry with that title.
For freeness and properness in Cartesian powers, see the start of~\S2.3 of \cite{grom:rigtran}.
P.~Olver points to \cite{olver:movfrmsing}
as a reference on freeness of prolonged group actions.
For applications, Olver recommends \cite{olver:lectmovfrm} and \cite{mansf:practicalguide}.
For work on transitive algebraic pseudogroup actions, he suggests \cite{kruglyc:lietresse}.

The arguments appearing in the present work are all elementary,
and do not lean on any stratified structure theory for real algebraic actions.
Some of the results of this paper likely extend to all local fields,
but we focus here on $\R$.
We use the language of algebraic geometry and the basics of real algebraic geometry,
but keep to a minimum what we require of the reader in Lie theoretic knowledge.

\section{Some terminology and basic facts\wrlab{sect-terminology}}

If $X$ is a set and if
$K_1,K_2,\ldots$ are subsets of $X$,
then we will say that a subset $S$ of $X$ is {\bf $K_\bullet$-small} if,
for some integer $j\ge1$, $S\subseteq K_j$.

Let a group $G$ act on a set $X$.
Let $G^\times:=G\backslash\{1_G\}$.
For all $x\in X$, $\Stab_G(x):=\{g\in G\,|\,gx=x\}$
will denote the stabilizer in $G$ of $x$, and the set $(\Stab_G(x))\backslash\{1_G\}$
will be called the {\bf$G^\times$-stabilizer of~$x$}.

Let a locally compact Hausdorff topological group $G$ act on a set~$X$.
Give $G^\times:=G\backslash\{1_G\}$ the relative topology inherited from $G$.
For all $x\in X$, let $G_x^\times:=(\Stab_G(x))\backslash\{1_G\}$
be the $G^\times$-stabilizer of $x$.
Let $S\subseteq X$.
We will say that {\bf$S$ has bounded $G^\times$-stabilizers}
if the union $\displaystyle{\bigcup_{s\in S}\,G_s^\times}$
is a precompact subset of~$G^\times$.

We will use the term {\bf real algebraic variety} to mean the $\R$-points of an $\R$-variety.
It is given its locally compact second countable Hausdorff topology
induced by the standard locally compact second countable Hausdorff topology on $\R$.

Throughout this note, the terms open, closed, locally compact, precompact, connected,
locally closed and constructible are all used to refer only to Hausdorff topologies.
Notions related to Zariski topologies  will be adorned with $\R$,
{\it e.g.}, $\R$-open, $\R$-closed, $\R$-closure, {\it etc.}

Let $Z$ be a real algebraic variety.
Then $Z$ will be called {\bf smooth}
if its $\R$-closure is a smooth $\R$-variety
all of whose irreducible components have the same dimension;
in this case, $Z$ is naturally a $C^\omega$ manifold.
A subset $S$ of $Z$ is {\bf real algebraic} if it is the $\R$-points of
an $\R$-closed subset of the $\R$-closure of $Z$; then $S$ is a closed subset of $Z$.
The collection of real algebraic subsets of $Z$ has the DCC under inclusion.
A subset of $Z$ is {\bf affine open} if it is the $\R$-points of an $\R$-open affine
subset of the $\R$-closure of~$Z$.
A subset $S$ of $Z$ is {\bf semialgebraic}
if, for every affine open subset $Z_0$ of $Z$,
the coordinate image (in real Euclidean space) of $S\cap Z_0$ is semialgebraic.
Semialgebraic subsets are always constructible.
The Tarski-Seidenberg Theorem implies that
any countable semialgebraic subset is finite.

Let $X$ and $Y$ be real algebraic varieties.
A function $f:X\to Y$ will be said to be {\bf real algebraic} if $f$ is the
$\R$-points of an $\R$-morphism from the $\R$-closure of $X$ to the $\R$-closure of $Y$.
A real algebraic function $f:X\to Y$ is always $C^0$ and, if $X$ and $Y$ are smooth, then $f$ is $C^\omega$.
If $f:X\to Y$ is real algebraic, then, for all $y\in Y$,
the fiber $f^{-1}(y)$ is a real algebraic subset of $X$.
The image, under a real algebraic function $f$, of any semialgebraic subset of $X$ is,
by the Tarski-Seidenberg Theorem, a semialgebraic subset of~$Y$.

Let $Z$ be a real algebraic variety and let $z\in Z$.
Give $Z\backslash\{z\}$ the relative topology inherited from $Z$.
Generally, $Z\backslash\{z\}$ is not a real algebraic subset of $Z$.
However, by algebraic localization, $Z\backslash\{z\}$
naturally has a real algebraic variety structure that is compatible with its relative topology
and such that the inclusion map $Z\backslash\{z\}\to Z$ is real algebraic.

We will use the term {\bf real algebraic group}
to mean the $\R$-points of an algebraic $\R$-group.
Such a group is naturally a $C^\omega$ Lie group.
A subgroup of a real algebraic group $G$ is a real algebraic subset
iff it is the $\R$-points of an $\R$-subgroup of the $\R$-closure of $G$.

Let $G$ be a real algebraic group.
Let $V$ be a real algebraic variety.
An action of $G$ on $V$ is said to be {\bf real algebraic} if it is
the $\R$-points of an $\R$-action of the $\R$-closure of $G$ on the $\R$-closure of $V$.
Such an action is always $C^0$, and, if $V$ is smooth, then it is a $C^\omega$ action.
If $G$ acts real algebraically on $V$, then, for all $v\in V$,
the stabilizer $\Stab_G(v)$ in $G$ of $v$ is a real algebraic subgroup of $G$.

\section{Some group theory and dynamics\wrlab{sect-gp-thy-dyn}}

\begin{lem}\wrlab{lem-loc-ctbl-implies-ctbl}
Let $G$ be a second countable topological group.
Let $S$ be a subgroup of $G$
such that, for some open neighborhood $U$ in $G$ of~$1_G$,
the set $S\cap U$ is countable.
Then $S$ is countable.
\end{lem}

\begin{proof}
For all $s\in S$, the intersection $S\cap(sU)=s(S\cap U)$ is countable.
Since $\{sU\,|\,s\in S\}$ is an open cover of $SU$,
and since $SU$ is Lindel\"of, choose a countable subset $C$ of $S$
such that $\displaystyle{\bigcup_{s\in C}\,sU=SU}$.
Then
$$S\quad=\quad S\cap(SU)\quad=\quad S\cap\left(\bigcup_{s\in C}\,sU\right)\quad=\quad\bigcup_{s\in C}\,(S\cap(sU)),$$
so $S$ is a countable union of countable sets,
so $S$ is countable.
\end{proof}

\begin{lem}\wrlab{lem-ctbl-implies-finite}
Let $G$ be a real algebraic group.
Then any countable real algebraic subgroup of $G$ is finite.
\end{lem}

\begin{proof}
Any real algebraic subgroup of $G$ is a semialgebraic subset of the real variety $G$.
It is a corollary of Tarski-Seidenberg that, in any real algebraic variety,
any countable semialgebraic subset is finite.
\end{proof}

\begin{lem}\wrlab{lem-bdd-punct-stab-free-on-open}
Let $G$ be a locally compact Hausdorff topological group
acting on a first countable topological space $Z$.
Assume the $G$-action on~$Z$ is $C^0$.
Let $G^\times:=G\backslash\{1\}$.
For all $z\in Z$, let $G_z^\times:=(\Stab_G(z))\backslash\{1\}$
be the $G^\times$-stabilizer of $z$.
Let $Z_1$ be an open subset of~$Z$.
Assume that $Z_1$ has bounded $G^\times$-stabilizers.
Let $W:=\{z\in Z_1\,|\,G_z^\times=\emptyset\}$.
Then $W$~is an open subset of $Z_1$.
\end{lem}

\begin{proof}
Say $z_1,z_2,\ldots\to z_\infty$ in~$Z_1$.
Assume, for all integers $j\ge1$, that $z_j\notin W$, {\it i.e.}, that $G_{z_j}^\times\ne\emptyset$.
We wish to show that $z_\infty\notin W$.

For all integers $j\ge1$, fix $g_j\in G_{z_j}^\times$.
Give $G^\times=G\backslash\{1_G\}$ the relative topology inherited from $G$.
Because $\displaystyle{\bigcup_{z\in Z_1}\,G_z^\times}$ is a precompact subset of~$G^\times$,
fix $g_\infty\in G^\times$ such that some subsequence of $g_1,g_2,\ldots$ converges in $G^\times$ to $g_\infty$.
For all integers $j\ge1$, we have $g_j\in G_{z_j}^\times\subseteq\Stab_G(z_j)$,
so $g_jz_j=z_j$.
It follows that $g_\infty z_\infty=z_\infty$.
Then $g_\infty\in\Stab_G(z_\infty)$.
Since $g_\infty\in G^\times$, we conclude that $g_\infty\ne1_G$.

Then $g_\infty\in(\Stab_G(z_\infty))\backslash\{1_G\}=G_{z_\infty}^\times$,
so $G_{z_\infty}^\times\ne\emptyset$, so $z_\infty\notin W$.
\end{proof}

\begin{lem}\wrlab{lem-free-proper-smooth}
Let $G$ be a $C^\omega$ Lie group acting on a $C^\omega$ manifold $N$.
Assume the $G$-action on $N$ is $C^\omega$, free and proper.
Then, for all $x\in N$, there exist
\begin{itemize}
\item a $G$-invariant open neighborhood $N_0$ in $N$ of $x$ and
\item a closed subset $T$ of $N_0$ (in the relative topology on $N_0$ inherited from $N$)
\end{itemize}
such that
\begin{itemize}
\item $T$ is a $C^\omega$ submanifold of $N_0$ and
\item the map $(g,t)\mapsto gt:G\times T\to N_0$ is a $C^\omega$ diffeomorphism.
\end{itemize}
\end{lem}

\begin{proof}
From Corollary B.30, p.~182, of \cite{ggk:momcobham},
we find that $N$ is a $C^\infty$ principal $G$-bundle.
So, since the $G$-action on $N$ is $C^\omega$,
it follows that $N$ is a $C^\omega$ principal $G$-bundle.
The result follows.
\end{proof}

\section{Some point-set topology\wrlab{sect-pt-set}}

Let $X$ be a set.
Let $K_1,K_2,\ldots$ be subsets of $X$.
Recall: A subset $S$ of $X$ is {\bf $K_\bullet$-small} if,
for some integer $j\ge1$, $S\subseteq K_j$.
Let $Y$ be a locally compact Hausdorff topological space.
Let $f:X\to Y$ be a function.

\begin{lem}\wrlab{lem-pt-set-ae-proper}
For all integers $j\ge1$, assume that $f(X\backslash K_j)$ is a constructible subset of \,$Y$.
Let $Y_0$ be an open subset of $Y$ such that $f^{-1}(Y_0)\ne\emptyset$.
For all $y\in Y_0$, suppose $f^{-1}(y)$ is $K_\bullet$-small.
Then there is an open subset $Y_1$ of $Y_0$
such that $f^{-1}(Y_1)$ is nonempty and $K_\bullet$-small.
\end{lem}

\begin{proof}
{\it Special Case:} Assume both that $Y_0=Y$ and that $f(X)$ is dense in $Y$.
{\it Proof in special case:}
For all integers $j\ge1$, let
$$C_j\quad:=\quad f(X\backslash K_j)\quad=\quad\{y\in Y\,|\,f^{-1}(y)\not\subseteq K_j\};$$
by assumption, $C_j$ is constructible in $Y$.
For all $y\in Y$, because $f^{-1}(y)$ is $K_\bullet$-small,
it follows, for some integer $j\ge1$, that $f^{-1}(y)\subseteq K_j$,
and, therefore, that $y\notin C_j$. That is, $C_1\cap C_2\cap\cdots=\emptyset$.

For all integers $j\ge1$, let $C_j^\circ$ be the interior in $Y$ of $C_j$.
By the Baire Category Theorem, in a locally compact Hausdorff topological space,
any countable intersection of dense open sets is nonempty.
So, since
$$C_1^\circ\cap C_2^\circ\cap\cdots\quad\subseteq\quad C_1\cap C_2\cap\cdots\quad=\quad\emptyset,$$
fix an integer $q\ge1$ such that $C_q^\circ$ is not dense in $Y$.
A constructible set in a topological space is dense iff its interior is dense.
Then $C_q$ is not dense in $Y$.
So let $Y_1$ be a nonempty open subset of $Y$ such that $C_q\cap Y_1=\emptyset$.
Then, for all $y\in Y_1$, we have $y\notin C_q$, so $f^{-1}(y)\subseteq K_q$.
Then $f^{-1}(Y_1)\subseteq K_q$, so $f^{-1}(Y_1)$ is $K_\bullet$-small.

As $Y_1$ is a nonempty open subset of $Y$,
and $f(X)$ is dense in~$Y$,
we get $(f(X))\cap Y_1\ne\emptyset$.
Then $f^{-1}(Y_1)\ne\emptyset$.
{\it End of proof in special case.}

{\it General case:}
Let $X':=f^{-1}(Y_0)$.
Give $Y_0$ the relative topology inherited from~$Y$.
Let $Y'$ denote the closure in $Y_0$ of $f(X')$.
Give $Y'$ the relative topology inherited from $Y_0$.
Then $f(X')$ is dense in~$Y'$.
Since $Y'$ is locally closed in~$Y$,
and since $Y$ is locally compact Hausdorff,
the topological space $Y'$ is locally compact Hausdorff, as well.
Because
$$X'\,\,=\,\,f^{-1}(f(X'))\,\,\subseteq\,\,f^{-1}(Y')\,\,\subseteq\,\,f^{-1}(Y_0)\,\,=\,\,X',$$
we conclude that $f^{-1}(Y')=X'$.

For all integers $j\ge1$, let $K'_j:=K_j\cap X'$;
then $X'\backslash K_j'=(X\backslash K_j)\cap X'$,
so, since $X'=f^{-1}(Y')$, it follows that
$f(X'\backslash K'_j)=(f(X\backslash K_j))\cap Y'$.
Then $f(X'\backslash K'_j)$ is a constructible subset of $Y'$.

Let $f':=f|X':X'\to Y'$. From the Special Case, replacing
\begin{itemize}
\item $X$ by $X'$,
\item $Y$ by $Y'$ \qquad and \qquad $Y_0$ also by $Y'$,
\item $Y_1$ by $Y'_1$,
\item $f:X\to Y$ by $f':X'\to Y'$ \qquad and
\item $K_1$ by $K'_1$, \qquad $K_2$ by $K'_2$, \qquad \dots
\end{itemize}
fix an open subset $Y'_1$ of the topological space $Y'$
such that $(f')^{-1}(Y'_1)$ is nonempty and $K'_\bullet$-small.
By definition of ``relative topology'',
fix an open subset $Y_1$ of $Y_0$ satisfying $Y'_1=Y'\cap Y_1$.
As $f^{-1}(Y')=X'$, we have
$f^{-1}(Y'_1)\,\,=\,\,(f^{-1}(Y'))\cap(f^{-1}(Y_1))\,\,=\,\,X'\cap(f^{-1}(Y_1))\,\,=\,\,f^{-1}(Y_1)$.
Then $f^{-1}(Y_1)=f^{-1}(Y'_1)=(f')^{-1}(Y'_1)$
is nonempty and $K'_\bullet$-small.
Any $K'_\bullet$-small set is $K_\bullet$-small,
so $f^{-1}(Y_1)$ is $K_\bullet$-small.
\end{proof}

\section{Some real algebraic geometry\wrlab{sect-real-alg}}

\begin{lem}\wrlab{lem-alg-implies-ae-proper}
Let $X$ and $Y$ be real algebraic varieties. Let $f:X\to Y$ be real algebraic.
Let $Y_0$ be an open subset of \,$Y$ such that, for all $y\in Y_0$, $f^{-1}(y)$ is a compact subset of $X$.
Assume that $f^{-1}(Y_0)\ne\emptyset$.
Then there exists an open subset $Y_1$ of $Y_0$
such that $f^{-1}(Y_1)$ is a nonempty precompact subset of $X$.
\end{lem}

\begin{proof}
Let $K_1,K_2,\ldots$ be precompact semialgebraic open subsets of $X$ such that
$K_1\subseteq K_2\subseteq\cdots$ and $K_1\cup K_2\cup\cdots=X$.
By Tarski-Seidenberg, for all integers $j\ge1$,
the set $f(X\backslash K_j)$ is semialgebraic in $Y$,
and is therefore constructible in $Y$.
A subset of $X$ is $K_\bullet$-small iff it is precompact in $X$,
so the result follows from \lref{lem-pt-set-ae-proper}.
\end{proof}

\section{Some continuous real algebraic dynamics\wrlab{sect-contin-real-alg-act}}

Let the multiplicative group $\R\backslash\{0\}$
act on~$\R^2\backslash\{(0,0)\}$ by the formula $t.(x,y)=(tx,y/t)$.
Thanks to D.~Witte Morris for pointing out to me that
this gives an example of a real algebraic action that is free, but not proper.
However, $\R^2\backslash(\R\times\{0\})$ and $\R^2\backslash(\{0\}\times\R)$
are both nonempty open invariant subsets,
and the action of $\R\backslash\{0\}$ on each of them {\it is} proper.
The next lemma (\lref{lem-cptstab-proper}) shows that {\it any} real algebraic action with
a lot of compact stabilizers has a lot of properness.

\begin{lem}\wrlab{lem-cptstab-proper}
Let $G$ be a real algebraic group acting on a real algebraic variety $V$.
Assume the $G$-action on $V$ is real algebraic.
Let $V_0$ be a nonempty $G$-invariant open subset of $V$.
Suppose, for all $v\in V_0$, that $\Stab_G(v)$ is a compact subset of $G$.
Then there exists a nonempty open $G$-invariant subset $U$ of $V_0$
such that the action of $G$ on~$U$ is proper.
\end{lem}

\begin{proof}
Let $X:=G\times V$ and let $Y:=V\times V$. Define a real algebraic map $f:X\to Y$ by
$f(g,v)=(v,gv)$.

{\it Claim:} For all $v,w\in V$,
if $\Stab_G(v)$ is compact and $y:=(v,w)\in Y$,
then $f^{-1}(y)$ is a compact subset of $X$.
{\it Proof of claim:}
If $w\notin Gv$, then
$$f^{-1}(y)\quad=\quad\{(g,v)\in X\,|\,gv=w\}\quad=\quad\emptyset,$$
so $f^{-1}(y)$ is a compact subset of $X$.
So we may assume that $w\in Gv$.
Fix $g\in G$ such that $w=gv$.
Then $y=(v,w)=(v,gv)=f(g,v)$. Let $S:=\Stab_G(v)$.
Then $S$ is compact. Then
\begin{eqnarray*}
f^{-1}(y)\,\,=\,\,f^{-1}(f(g,v))&=&\{(g',v')\in X\,|\,v=v',gv=g'v'\}\\
&=&\{(g',v)\in X\,|\,v=g^{-1}g'v\}\\
&=&(gS)\times\{v\}
\end{eqnarray*}
is compact as well.
{\it End of proof of claim.}

Let $Y_0:=V_0\times V_0$.
By assumption, for all $v\in V_0$, $\Stab_G(v)$ is compact.
So, by the Claim, it follows,
for all $y\in Y_0=V_0\times V_0$, that $f^{-1}(y)$ is a compact subset of $X$.
By assumption, $V_0\ne\emptyset$.
For all $v\in V_0$, we have $f(e,v)=(v,v)\in Y_0$, so $(e,v)\in f^{-1}(Y_0)$.
Then $f^{-1}(Y_0)\ne\emptyset$.

By \lref{lem-alg-implies-ae-proper},
fix an open subset $Y_1$ of $Y_0$
such that $f^{-1}(Y_1)$ is a nonempty precompact subset of~$X$.
As $f^{-1}(Y_1)\ne\emptyset$,
fix $g_0\in G$ and $v_0\in V$ such that $f(g_0,v_0)\in Y_1$.
Then
$$(v_0,g_0v_0)\,\,=\,\,f(g_0,v_0)\,\,\in\,\,Y_1\,\,\subseteq\,\,Y_0\,\,=\,\,V_0\times V_0,$$
so $v_0\in V_0$, $g_0v_0\in V_0$, and
$Y_1$ is an open neighborhood in $Y_0=V_0\times V_0$ of $(v_0,g_0v_0)$.
Choose open neighborhoods $A$ and $B$ in $V_0$ of $v_0$ and $g_0v_0$, respectively,
such that $A\times B\subseteq Y_1$.
Let $C:=A\cap(g_0^{-1}B)$.

Let $U:=GC$.
As $v_0\in A$ and $g_0v_0\in B$,
we get $v_0\in A\cap(g_0^{-1}B)=C$, so $C\ne\emptyset$.
Then $U\ne\emptyset$.
Since $A$ and $B$ are open subsets of $V_0$,
it follows that $C=A\cap(g_0^{-1}B)$ is an open subset of $V_0$.
So, as $V_0$ is $G$-invariant,
we see that $U=GC$ is a $G$-invariant open subset of $V_0$.
It remains to show that the action of $G$ on $U$ is proper.

Let $K$ be a compact subset of $U$.
We wish to show that
$$G'\quad:=\quad\{g\in G\,|\,(gK)\cap K\ne\emptyset\}$$
is a compact subset of $G$.
By compactness of $K$ and by continuity of the $G$-action on $V$,
it follows that $G'$ is a closed subset of $G$,
so it suffices to show that $G'$ is precompact in $G$.

Since $A\times B\subseteq Y_1$, we have $f^{-1}(A\times B)\subseteq f^{-1}(Y_1)$,
so, as $f^{-1}(Y_1)$ is precompact in~$X$,
we see that $f^{-1}(A\times B)$ is precompact in $X$, {\it i.e.}, in $G\times V$.
Let $\pi:G\times V\to G$ be projection onto the first coordinate.
Then $P:=\pi(f^{-1}(A\times B))$ is precompact in $G$.
Moreover,
\begin{eqnarray*}
P\,\,=\,\,\pi(f^{-1}(A\times B))&=&\pi(\{(g,v)\in G\times V\,|\,v\in A,gv\in B\})\\
&=&\{g\in G\,|\,\exists v\in A~\hbox{s.t.}~gv\in B\}\\
&=&\{g\in G\,|\,(gA)\cap B\ne\emptyset\}
\end{eqnarray*}

Since $C\subseteq A$, it follows,
for all $g\in G$, that $gC\subseteq gA$.
So, since $C\subseteq g_0^{-1}B$,
we see, for all $g\in G$, that $(gC)\cap C\subseteq(gA)\cap(g_0^{-1}B)$.
Then
\begin{eqnarray*}
G''&:=&\{g\in G\,|\,(gC)\cap C\ne\emptyset\}\\
&\subseteq&\{g\in G\,|\,(gA)\cap(g_0^{-1}B)\ne\emptyset\}\\
&=&\{g\in G\,|\,(g_0gA)\cap B\ne\emptyset\}\\
&=&g_0^{-1}\{g\in G\,|\,(gA)\cap B\ne\emptyset\}\,\,=\,\,g_0^{-1}P.
\end{eqnarray*}
So, since $P$ is precompact in $G$, it follows that $G''$ is precompact in $G$.

Since $\{gC\,|\,g\in G\}$ is an open cover of $GC=U$,
and since $K$ is a compact subset of $U$,
choose an integer $n\ge1$ and choose $g_1,\ldots,g_n\in G$
such that $K\subseteq(g_1C)\cup\cdots\cup(g_nC)$.
Let $C':=(g_1C)\cup\cdots\cup(g_nC)$.
Then
\begin{eqnarray*}
G'&=&\{g\in G\,|\,(gK)\cap K\ne\emptyset\}\\
&\subseteq&\{g\in G\,|\,(gC')\cap C'\ne\emptyset\}\\
&=&\,\bigcup_{j,k}\,\,\{g\in G\,|\,(gg_jC)\cap(g_kC)\ne\emptyset\}\\
&=&\,\bigcup_{j,k}\,\,\{g\in G\,|\,(g_k^{-1}gg_jC)\cap C\ne\emptyset\}\\
&=&\,\bigcup_{j,k}\,\,g_k\,\{g\in G\,|\,(gC)\cap C\ne\emptyset\}\,g_j^{-1}\\
&=&\,\bigcup_{j,k}\,\,g_k\,G''\,g_j^{-1}.
\end{eqnarray*}
So, since $G''$ is precompact in $G$, $G'$ is also precompact in $G$.
\end{proof}

Maintaining the notation of \lref{lem-cptstab-proper},
if $\scrw$ is a collection of pairwise disjoint $G$-invariant open subsets of $V_0$,
and if, for all $W\in\scrw$, the $G$-action on $W$ is proper,
then the $G$-action on the union $\cup\,\scrw$ is also proper.
Using this, it is not hard to improve \lref{lem-cptstab-proper},
and show that the subset $U$ may be chosen to be dense in $V_0$.
However, using such an argument may produce a complicated set $U$.
For example, even if $V_0$ is connected,
the set $U$ might, in principle, have infinitely many connected components.
Assuming that $V\backslash V_0$ is a real algebraic subset of $V$,
it would be interesting to know if $U$ can always be chosen so that
$V\backslash U$ is also a real algebraic subset of~$V$.

Let $\scru$ denote the collection of
open subsets $U$ of $V_0$ such that $G$ acts freely and properly on $U$.
It would also be useful to know whether $\scru$
is necessarily an open cover of $V_0$, {\it i.e.}, whether $\cup\,\scru=V_0$.

\begin{lem}\wrlab{lem-force-bdd-punct-stabs}
Let $G$ be a real algebraic group acting on a real algebraic variety $Y$.
Assume the $G$-action on $Y$ is real algebraic.
Let $Y_0$ be a nonempty open subset of $Y$.
Assume, for all $y\in Y_0$, that $\Stab_G(y)$ is finite.
Let $G^\times:=G\backslash\{1\}$.
Then there exists a nonempty open subset $Y_1$ of $Y_0$ such that
$Y_1$ has bounded $G^\times$-stabilizers.
\end{lem}

\begin{proof}
For all $y\in Y$, let $G_y^\times:=(\Stab_G(y))\backslash\{1\}$
be the $G^\times$-stabilizer of~$y$.
If, for all $y\in Y_0$, we have
$G_y^\times=\emptyset$,
then, setting $Y_1$ equal to~$Y_0$, we are done.
We therefore assume, for some $y_0\in Y_0$, that $G_{y_0}^\times\ne\emptyset$.

Give $G^\times=G\backslash\{1\}$ the relative topology inherited from~$G$.
By algebraic localization, $G^\times$ naturally has a real algebraic variety structure
that is compatible with its relative topology
and such that the inclusion map $G^\times\to G$ is real algebraic. Let
$$X\quad:=\quad\{\,(g,y)\in G^\times\times Y\,\,|\,\,gy=y\,\}.$$
Define $\pi:X\to G^\times$ and $f:X\to Y$ by $\pi(g,y)=g$ and $f(g,y)=y$.
For all $y\in Y$, we have $f^{-1}(y)=G_y^\times\times\{y\}$.
For some $y_0\in Y_0$, we have $G_{y_0}^\times\ne\emptyset$, so $f^{-1}(y_0)=G_{y_0}^\times\times\{y_0\}\ne\emptyset$.
It follows that $f^{-1}(Y_0)\ne\emptyset$.

For all $y\in Y_0$, $\Stab_G(y)$ is finite,
so $G_y^\times$ is finite, so $f^{-1}(y)=G_y^\times\times\{y\}$ is a finite,
and therefore compact, subset of $X$.
By \lref{lem-alg-implies-ae-proper},
fix an open subset $Y_1$ of~$Y_0$
such that $f^{-1}(Y_1)$ is a nonempty precompact subset of~$X$.
Then $\pi(f^{-1}(Y_1))$ is a precompact subset of $G^\times$.
So, since $\displaystyle{\pi(f^{-1}(Y_1))=\bigcup_{y\in Y_1}\,G_y^\times}$,
we see that $Y_1$ has bounded $G^\times$-stabilizers.
Because $f^{-1}(Y_1)\ne\emptyset$, it follows that $Y_1\ne\emptyset$.
\end{proof}

The next lemma (\lref{lem-loc-free-on-open}) follows from
semicontinuity of connected stabilizers, which is well-known.
However, to minimize the Lie theoretic requirements on the reader,
we provide a proof.

\begin{lem}\wrlab{lem-loc-free-on-open}
Let $G$ be a real algebraic group acting on a real algebraic variety $Y$.
Assume the $G$-action on $Y$ is real algebraic.
For all $y\in Y$, let $G_y:=\Stab_G(y)$.
Let $Y':=\{y\in Y\,|\,G_y\hbox{ is finite}\}$.
Then $Y'$~is an open subset of $Y$.
\end{lem}

\begin{proof}
Say $y_1,y_2,\ldots\to y_\infty$ in~$Y$.
Assume, for all integers $i\ge1$, that $y_i\notin Y'$,
{\it i.e.}, that $G_{y_i}$ is infinite.
We wish to show that $y_\infty\notin Y'$.

Let $\Lg$ be the Lie algebra of~$G$.
Fix a positive definite quadratic form $Q:\Lg\to[0,\infty)$.
For all~$X\in\Lg$, let $|X|:=\sqrt{Q(X)}$.
For all $\varepsilon>0$, let ${\overline{B}}_\varepsilon:=Q^{-1}([0,\varepsilon^2])$
be the closed ball in $\Lg$ centered at $0_\Lg$ with radius~$\varepsilon$.
Multiplying $Q$ by a scalar if necessary, we may assume that the exponential map
$\exp:\Lg\to G$ is one-to-one on ${\overline{B}}_1$.

For every integer $j\ge1$, let $\delta_j:=1/j$.
For all integers $j\ge1$, let ${\overline{A}}_j:=Q^{-1}([\delta_{2j+1}^2,\delta_{2j}^2])$
be the closed annulus in $\Lg$ centered at $0_\Lg$ with radii $\delta_{2j+1}$ and $\delta_{2j}$.
Since $1=\delta_1>\delta_2>\cdots$,
it follows that the sets ${\overline{A}}_1,{\overline{A}}_2,\ldots$ are pairwise disjoint,
and that ${\overline{A}}_1\cup{\overline{A}}_2\cup\cdots\subseteq{\overline B}_1$.

For all integers $i\ge1$, let
$\Lg_i:=\{X\in\Lg\,|\,\exp(X)\in G_{y_i}\}$.
For all $\varepsilon>0$, $\exp({\overline{B}}_\varepsilon)$ contains an open neighborhood in $G$ of $1$.
For all integers $i\ge1$, because $G_{y_i}$ is infinite,
we see (by \lref{lem-ctbl-implies-finite}) that $G_{y_i}$ is uncountable,
and then (by \lref{lem-loc-ctbl-implies-ctbl}) that, for all $\varepsilon>0$,
the set $G_{y_i}\cap(\exp({\overline{B}}_\varepsilon))$ is uncountable,
which implies that $\Lg_i\cap{\overline{B}}_\varepsilon$ is uncountable.

For every integer $j\ge1$, let $\varepsilon_j:=\delta_{2j}-\delta_{2j+1}$.
For all integers $i,j\ge1$, $\Lg_i\cap{\overline{B}}_{\varepsilon_j}$ is uncountable,
so fix $R_{ij}\in\Lg_i\cap{\overline{B}}_{\varepsilon_j}$ such that $R_{ij}\ne0$; then
$0<|R_{ij}|\le\varepsilon_j<\delta_{2j}$,
so fix an integer $n_{ij}\ge1$ such that
$$n_{ij}|R_{ij}|\le\delta_{2j}\qquad\hbox{and}\qquad(n_{ij}+1)|R_{ij}|\ge\delta_{2j}.$$
For all integers $i,j\ge1$, let $S_{ij}:=n_{ij}R_{ij}$.

Let $i,j\ge1$ be integers.
Then $|S_{ij}|=n_{ij}|R_{ij}|\le\delta_{2j}$.
Also, because $R_{ij}\in{\overline{B}}_{\varepsilon_j}$,
it follows that $\varepsilon_j\ge|R_{ij}|$, so
$$|S_{ij}|+\varepsilon_j\quad\ge\quad n_{ij}|R_{ij}|+|R_{ij}|\quad=\quad(n_{ij}+1)|R_{ij}|\quad\ge\quad\delta_{2j},$$
so $|S_{ij}|\ge\delta_{2j}-\varepsilon_j=\delta_{2j+1}$.
Then $\delta_{2j+1}\le|S_{ij}|\le\delta_{2j}$,
so $S_{ij}\in{\overline{A}}_j$.

Let $i,j\ge1$ be integers.
Since $R_{ij}\in\Lg_i$, we have $\exp(R_{ij})\in G_{y_i}$.
Then $\exp(n_{ij}R_{ij})=(\exp(R_{ij}))^{n_{ij}}\in G_{y_i}$.
So, as $S_{ij}=n_{ij}R_{ij}$, we have $\exp(S_{ij})\in G_{y_i}$.
That is, $(\exp(S_{ij}))y_i=y_i$.

Let $j\ge1$ be an integer.
Then $S_{1j},S_{2j},\ldots$ is a sequence in the compact set ${\overline{A}}_j$,
so choose $S_{\infty j}\in{\overline{A}}_j$ such that some subsequence of
$S_{1j},S_{2j},\ldots$ converges to $S_{\infty j}$.
For all integers $i\ge1$, $(\exp(S_{ij}))y_i=y_i$,
and it follows that $(\exp(S_{\infty j}))y_\infty=y_\infty$,
{\it i.e.}, that $\exp(S_{\infty j})\in G_{y_\infty}$.

The sets ${\overline{A}}_1,{\overline{A}}_2,\ldots$ are pairwise disjoint,
so $S_{\infty1},S_{\infty2},\ldots$ are distinct.
So, as $\exp$ is one-to-one on~${\overline{B}}_1$ and
$S_{\infty1},S_{\infty2},\ldots\in{\overline{A}}_1\cup{\overline{A}}_2\cup\cdots\subseteq{\overline{B}}_1$,
we see that $\exp(S_{\infty 1}),\exp(S_{\infty 2}),\ldots$ are distinct.
Since they are all elements of $G_{y_\infty}$,
we conclude that $G_{y_\infty}$ is infinite, so $y_\infty\notin Y'$.
\end{proof}

\section{Some smooth real algebraic dynamics\wrlab{sect-real-alg-act}}

Let $G$ be a real algebraic group and let $G^\circ$ be an open subgroup of~$G$.
Then $G$ and $G^\circ$ are $C^\omega$ Lie groups.
An open subgroup of a topological group is closed, so $G^\circ$ is a closed subgroup of~$G$.
Let $1:=1_G=1_{G^\circ}$ be the identity element in $G$ and in $G^\circ$.

Let $G$ act on a smooth real algebraic variety $M$.
Assume that the $G$-action on $M$ is real algebraic.
Let $M^\circ$ be a nonempty $G^\circ$-invariant open subset of~$M$.
Then $M$ and $M^\circ$ are $C^\omega$ manifolds
and the actions of $G$ on~$M$ and of $G^\circ$ on $M^\circ$ are $C^\omega$ actions.

Let $G$ act on a sequence of smooth real algebraic varieties $M_1,M_2,\ldots$.
Assume that these $G$-actions are all real algebraic.
Let $M_0:=M$.
Let
$$\cdots\quad\to\quad M_2\quad\to\quad M_1\quad\to\quad M_0\quad=\quad M$$
be a sequence of $G$-equivariant real algebraic maps.
(In the basic example we have in mind, $M_i$ is the $i$-th order frame bundle,
but other jet bundle towers will also satisfy the assumptions listed below.)

For all integers $i>0$, let $M_i^\circ$ be the preimage of $M^\circ$ under the composite
$M_i\to\cdots\to M_0=M$. Let $M_0^\circ:=M^\circ$.
Then, for all integers $i\ge0$, $M_i^\circ$ is an open $G^\circ$-invariant subset of $M_i$.

For all integers $i\ge0$, for all integers $j>i$,
let $\pi_i^j:M_j^\circ\to M_i^\circ$ be
the restriction to $M_j^\circ$ of the composite $M_j\to\cdots\to M_i$.
For all integers $i\ge0$, let $\pi_i^i:M_i^\circ\to M_i^\circ$ be the identity map.

Throughout this section, we assume that $M_0,M_1,M_2,\ldots$ are pairwise disjoint sets.
Let $M_\bullet^\circ:=M_0^\circ\cup M_1^\circ\cup M_2^\circ\cup\cdots$.
For all integers $i\ge0$, for all $x\in M_i^\circ$, let
$\displaystyle{F_x^\circ:=\bigcup_{j=i}^\infty\,\,(\pi_i^j)^{-1}(x)}$.

Throughout this section, we assume, for all $x\in M_\bullet^\circ$,
that the action of $\Stab_{G^\circ}(x)$ on $F_x^\circ$ is effective.
When $M^\circ$ is connected and $G^\circ$ acts effectively on $M^\circ$,
this effectiveness condition is satisfied if
the tower $\cdots\to M_2\to M_1\to M_0=M$ is any typical tower of jet bundles over~$M$.

\begin{lem}\wrlab{lem-triv-stab-in-prolong}
$\forall u\in M_\bullet^\circ$,
$\exists v\in F_u^\circ$ such that $\Stab_{G^\circ}(v)=\{1\}$.
\end{lem}

\begin{proof}
The set of real algebraic subgroups of $G$ has the DCC (under inclusion),
so fix a minimal element $S$ in $\{\Stab_G(x)\,|\,x\in F_u^\circ\}$.
Fix $v\in F_u^\circ$ such that $\Stab_G(v)=S$.
We will show that $\Stab_{G^\circ}(v)=\{1\}$.
Let $S^\circ:=G^\circ\cap S=\Stab_{G^\circ}(v)$.
Let $s\in S^\circ$.
We wish to prove $s=1$.

Since the action of $S^\circ=\Stab_{G^\circ}(v)$ on~$F_v^\circ$ is effective,
it suffices to show, for all $w\in F_v^\circ$, that $sw=w$.
Let $w\in F_v^\circ$.
We wish to prove that $sw=w$, {\it i.e.}, that $s\in\Stab_G(w)$.

As $\Stab_G(w)\subseteq\Stab_G(v)=S$
and $w\in F_v^\circ\subseteq F_u^\circ$, we see, by minimality of $S$,
that $\Stab_G(w)=S$.
Then $s\in S^\circ\subseteq S=\Stab_G(w)$.
\end{proof}

\begin{cor}\wrlab{cor-triv-stab-above}
$\exists v\in M_\bullet^\circ$ such that $\Stab_{G^\circ}(v)=\{1\}$.
\end{cor}

\begin{proof}
Fix any $u\in M_\bullet^\circ$.
Then $F_u^\circ\subseteq M_\bullet^\circ$.
Apply \lref{lem-triv-stab-in-prolong}.
\end{proof}

\begin{lem}\wrlab{lem-free-on-open}
For some integer $l\ge0$, there exists a nonempty open subset $W$ of $M_l^\circ$
such that, for all $w\in W$, we have $\Stab_{G^\circ}(w)=\{1\}$.
\end{lem}

\begin{proof}
By \cref{cor-triv-stab-above} (with $v$ replaced by~$q$),
fix $q\in M_\bullet^\circ$ such that $\Stab_{G^\circ}(q)=\{1\}$.
Fix an integer $k\ge1$ such that $q\in M_k^\circ$.
Let $Y:=M_k$.
For all $y\in Y$, let $G_y:=\Stab_G(y)$.

We have $G_q\cap G^\circ=\Stab_{G^\circ}(q)=\{1\}$,
so $G_q\cap G^\circ$ is countable.
So, as $G^\circ$ is an open neighborhood in $G$ of $1$,
we see (by \lref{lem-loc-ctbl-implies-ctbl}) that $G_q$ is countable,
and (by \lref{lem-ctbl-implies-finite}) that $G_q$ is finite.

Let $Y':=\{y\in Y\,|\,G_y\hbox{ is finite}\}$
and $Y_0:=M_k^\circ\cap Y'$.
Then $q\in Y_0$, so $Y_0\ne\emptyset$.
By \lref{lem-loc-free-on-open}, $Y'$ is an open subset of $Y$,
so $Y_0$ is an open subset of $M_k^\circ$.
Since $Y_0$ is open in $M_k^\circ$, and since $M_k^\circ$ is open in $M_k=Y$,
it follows that $Y_0$ is open in $Y$.
As $Y_0\subseteq Y'$, we see, for all $y\in Y_0$, that $G_y=\Stab_G(y)$ is finite.
Let $G^\times:=G\backslash\{1\}$ and $G^{\circ\times}:=G^\circ\backslash\{1\}$.
By \lref{lem-force-bdd-punct-stabs},
choose a nonempty open subset $Y_1$ of $Y_0$
such that $Y_1$ has bounded $G^\times$-stabilizers.
Then, as $G^\circ$ is a closed subgroup of~$G$,
it follows that $Y_1$ has bounded $G^{\circ\times}$-stabilizers.

Fix $u\in Y_1$.
By \lref{lem-triv-stab-in-prolong},
fix $v\in F_u^\circ$ such that $\Stab_{G^\circ}(v)=\{1\}$.
Fix an integer $l\ge k$ such that
$v\in(\pi_k^l)^{-1}(u)$.
Let $\pi:=\pi_k^l:M_l^\circ\to M_k^\circ$.
Since $Y_1$ is open in $Y_0$, and since $Y_0$ is open in $M_k^\circ$,
it follows that $Y_1$~is open in $M_k^\circ$.
Let $Z_1:=\pi^{-1}(Y_1)$.
Then $Z_1$ is open in $Z:=M_l^\circ$.
As $Y_1$ has bounded $G^{\circ\times}$-stabilizers,
$Z_1$ also has bounded $G^{\circ\times}$-stabilizers.

For all $z\in Z$, let
$G_z^{\circ\times}:=(\Stab_{G^\circ}(z))\backslash\{1\}$
be the $G^{\circ\times}$-stabilizer of~$z$.
Let $W:=\{z\in Z_1\,|\,G_z^{\circ\times}=\emptyset\}$.
By \lref{lem-bdd-punct-stab-free-on-open} (with $G$ replaced by~$G^\circ$),
$W$ is an open subset of $Z_1$.
Since $W$ is open in $Z_1$, and since $Z_1$ is open in $Z$,
it follows that $W$ is open in $Z=M_l^\circ$.

We have $v\in(\pi_k^l)^{-1}(u)=\pi^{-1}(u)\subseteq\pi^{-1}(Y_1)=Z_1$.
Also, we know that $\Stab_{G^\circ}(v)=\{1\}$, {\it i.e.}, that $G_v^{\circ\times}=\emptyset$.
Then $v\in W$, so $W\ne\emptyset$.
By definition of $W$, for all $w\in W$, $G_w^{\circ\times}=\emptyset$,
{\it i.e.}, $\Stab_{G^\circ}(w)=\{1\}$.
\end{proof}

\begin{prop}\wrlab{prop-free-proper-in-prolong}
For some integer $l\ge0$,
there exists a nonempty $G^\circ$-invariant open subset $N$ of $M_l^\circ$
such that the action of $G^\circ$ on $N$ is free and proper.
\end{prop}

\begin{proof}
By \lref{lem-free-on-open}, fix an integer $l\ge0$
and a nonempty open subset $W$ of $M_l^\circ$
such that, for all $w\in W$, $\Stab_{G^\circ}(w)=\{1\}$.
Then $G^\circ W$ is a $G^\circ$-invariant open subset of $M_l^\circ$,
and the $G^\circ$-action on $G^\circ W$ is free.

Let $V:=M_l$. For all $v\in V$, let $G_v:=\Stab_G(v)$.
For all $w\in W$, we have $G_w\cap G^\circ=\Stab_{G^\circ}(w)=\{1\}$,
so, since $G^\circ$ is an open neighborhood in $G$ of $1$,
we conclude (by \lref{lem-loc-ctbl-implies-ctbl}) that $G_w$ is countable,
hence (by \lref{lem-ctbl-implies-finite}) finite.
Let $V_0:=GW$.
Then, for all $v\in V_0$, the stabilizer $G_v=\Stab_G(v)$ is finite, hence compact.
Since $W$ is open in $M_l^\circ$ and since $M_l^\circ$ is open in $M_l=V$,
it follows that $W$ is open in $V$.
Then $V_0=GW$ is a $G$-invariant open subset of $V$.
Since $W\ne\emptyset$, we have $V_0=GW\ne\emptyset$.
By \lref{lem-cptstab-proper}, choose a nonempty $G$-invariant open subset $U$ of~$V_0$
such that the action of $G$ on $U$ is proper.

Since $U$ is a nonempty $G$-invariant subset of $V_0=GW$, it follows that $U\cap W\ne\emptyset$.
Then $N:=U\cap(G^\circ W)$ is nonempty and $G^\circ$-invariant.
Since $U$ is open in $V_0$ and since $V_0$ is open in $V$, it follows that $U$ is open in $V=M_l$.
So, since $G^\circ W$ is open in $M_l^\circ$,
we conclude that $N=U\cap(G^\circ W)$ is open in $M_l^\circ$.

The $G^\circ$-action on~$G^\circ W$ is free,
and $N$ is a $G^\circ$-invariant subset of~$G^\circ W$.
Therefore, the $G^\circ$-action on~$N$ is free.

Because the $G$-action on $U$ is proper,
and because $G^\circ$ is a closed subgroup of $G$,
we see that the $G^\circ$-action on $U$ is proper, as well.
Then, because $N$ is a $G^\circ$-invariant subset of $U$,
it follows that the $G^\circ$-action on $N$ is also proper.
\end{proof}

\begin{cor}\wrlab{cor-free-proper-in-prolong}
For some integer $l\ge0$, there exist
\begin{itemize}
\item a nonempty $G^\circ$-invariant open subset $N_0$ of $M_l^\circ$ and
\item a closed subset $T$ of $N_0$ (in the relative topology on $N_0$ inherited from $M_l^\circ$)
\end{itemize}
such that
\begin{itemize}
\item $T$ is a $C^\omega$ submanifold of $N_0$ and
\item the map $(g,t)\mapsto gt:G^\circ\times T\to N_0$ is a $C^\omega$ diffeomorphism.
\end{itemize}
\end{cor}

\begin{proof}
Choose $l$ and $N$ as in \pref{prop-free-proper-in-prolong}.
The $G^\circ$-action on $N$ is then $C^\omega$, free and proper.
Fix any $x\in N$.
The result then follows from \lref{lem-free-proper-smooth} (with $G$ replaced by $G^\circ$).
\end{proof}

The pair $(T,N_0)$ above is sometimes called a ``moving frame''.


\bibliography{list}

\end{document}